\documentclass[12pt]{amsart}
\usepackage{amssymb}
\usepackage{amsfonts}
\usepackage{amssymb,latexsym}
\usepackage{enumerate}
\usepackage{mathrsfs}
\usepackage{hyperref}
 \usepackage{color}\makeatletter
\@namedef{subjclassname@2010}{%
  \textup{2010} Mathematics Subject Classification}
\makeatother

  [2002/01/19 v2.2g %
    AMS font definitions%
  ]
\DeclareFontFamily{U}{euf}{}
\DeclareFontShape{U}{euf}{m}{n}{%
  <5><6><7><8><9>gen*eufm%
  <10><10.95><12><14.4><17.28><20.74><24.88>eufm10%
  }{}
\DeclareFontShape{U}{euf}{b}{n}{%
  <5><6><7><8><9>gen*eufb%
  <10><10.95><12><14.4><17.28><20.74><24.88>eufb10%
  }{}

  [2002/01/19 v2.2g %
    AMS font definitions%
  ]
\DeclareFontFamily{U}{msb}{}
\DeclareFontShape{U}{msb}{m}{n}{%
  <5><6><7><8><9>gen*msbm%
  <10><10.95><12><14.4><17.28><20.74><24.88>msbm10%
  }{}

  [2002/01/19 v2.2g %
    AMS font definitions%
  ]
\DeclareFontFamily{U}{msa}{}
\DeclareFontShape{U}{msa}{m}{n}{%
  <5><6><7><8><9>gen*msam%
  <10><10.95><12><14.4><17.28><20.74><24.88>msam10%
  }{}

\newtheorem{theorem}{Theorem}[section]
\newtheorem{lemma}[theorem]{Lemma}
\newtheorem{proposition}[theorem]{Proposition}

\theoremstyle{definition}
\newtheorem{definition}[theorem]{Definition}

\newtheorem{remark}[theorem]{Remark}

\numberwithin{equation}{section} 

\def\deg{\textrm{deg}}

\begin{document}

\title[Infinite order linear difference equation]
{Infinite order linear difference equation satisfied  by  a refinement of Goss zeta function}



\author{Su Hu}
\address{Department of Mathematics, South China University of Technology, Guangzhou, Guangdong 510640, China}
\email{mahusu@scut.edu.cn}

\author{Min-Soo Kim}
\address{Department of Mathematics Education, Kyungnam University, Changwon, Gyeongnam 51767, Republic of Korea}
\email{mskim@kyungnam.ac.kr}




\begin{abstract}
At the international congress of mathematicians in 1900, Hilbert  claimed that the Riemann zeta function $\zeta(s)$
is not the solution of any algebraic ordinary differential equations on its region of analyticity. Let $T$ be
an infinite order linear differential operator  introduced by Van Gorder in 2015. Recently, Prado and Klinger-Logan \cite{PK} showed that the Hurwitz zeta function $\zeta(s,a)$ formally satisfies the following linear  differential
equation
$$
    T\left[\zeta (s,a) - \frac{1}{a^s}\right] = \frac{1}{(s-1)a^{s-1}}.
$$
 Then in \cite{HK2021}, by defining $T_{p}^{a}$, a $p$-adic analogue of Van Gorder's operator $T,$  we constructed the following convergent  infinite order linear differential   equation satisfied by the $p$-adic
 Hurwitz-type Euler zeta function $\zeta_{p,E}(s,a)$
$$
T_{p}^{a}\left[\zeta_{p,E}(s,a)-\langle a\rangle^{1-s}\right]
=\frac{1}{s-1}\left(\langle a-1 \rangle^{1-s}-\langle a\rangle^{1-s}\right).
$$
 
In this paper, we consider this problem in the positive characteristic case.
That is, by introducing  $\zeta_{\infty}(s_{0},s,a,n)$, a Hurwitz type refinement of Goss zeta function, and an infinite order linear difference operator $L$,
we establish  the following difference equation
\begin{equation*} L\left[\zeta_{\infty}\left(\frac{1}{T},s,a,0\right)\right]=\sum_{\gamma\in\mathbb{F}_{q}} \frac{1}{\langle a+\gamma\rangle^{s}}.
\end{equation*}
 
 \end{abstract}

\maketitle

\section{Introduction}
Let 
\begin{equation}\label{Riemann}
\zeta(s)=\sum_{n=1}^{\infty}\frac{1}{n^{s}},\quad\textrm{Re}(s) > 1
\end{equation}
be the Riemann zeta function. 
At the international congress of mathematicians in 1900, Hilbert \cite{HilbertProb} claimed that  $\zeta(s)$
is not the solution of any algebraic ordinary differential equations on its region of analyticity.
In 2015, Van Gorder \cite{VanGorder}  showed that $\zeta(s)$  formally satisfies an infinite order linear differential equation
\begin{equation}\label{zetaDE}
T[\zeta(s)-1]=\frac{1}{s-1}.
\end{equation}
For $0 < a \leq 1$, Re$(s) >1$, let
\begin{equation}~\label{Hurwitz}
\zeta(s,a)=\sum_{n=0}^{\infty}\frac{1}{(n+a)^{s}},
\end{equation}
be the Hurwitz zeta function  (see \cite{Hurwitz}).
This function can  be analytically continued to a meromorphic function in the
complex plane with a simple pole at $s=1$.
In 2020,  Prado and Klinger-Logan \cite{PK}
 showed that the Hurwitz zeta function $\zeta(s,a)$ also formally satisfies an infinite order linear differential equation
under the same operator $T$
\begin{equation}\label{HurDE1}
    T\left[\zeta (s,a) - \frac{1}{a^s}\right] = \frac{1}{(s-1)a^{s-1}}
\end{equation}
for $s \in \mathbb{C}$ satisfying $s + n \neq 1$ for all $n \in \mathbb{Z}_{\geq 0}.$  
But unfortunately, in the same paper they proved that
the operator $T$ applied to Hurwitz zeta function, does not converge at any point in the complex plane $\mathbb{C}$ (see \cite[Theorem 8]{PK}).

The Hurwitz-type Euler zeta function
\begin{equation*}\label{HEZ}
\zeta_E(s,a)=\sum_{n=0}^\infty\frac{(-1)^n}{(n+a)^s}.
\end{equation*}
is an alternating form of the Hurwitz zeta function $\zeta(s,a)$.
This  function can  be analytically
continued to the complex plane without any pole.
Let $\zeta_{p,E}(s,a)$ be its $p$-adic analogue,
which is defined from the integral transform
\begin{equation}\label{E-zeta}
\zeta_{p,E}(s,a)=\int_{\mathbb{Z}_{p}}\langle a+t\rangle^{1-s}d\mu_{-1}(t)
\end{equation}
for $a\in \mathbb{C}_{p}\backslash \mathbb{Z}_{p}$,
where \begin{equation}\label{-q-e2}
I_{-1}(f)=\int_{\mathbb Z_p}f(t)d\mu_{-1}(t)
=\lim_{r\rightarrow\infty}\sum_{k=0}^{p^r-1}f(k)(-1)^k.
\end{equation}
is the  fermionic $p$-adic integral for a continuous function $f$ on $\mathbb{Z}_{p}$. 
 For details  on  the definition and the properties of  $\zeta_{p,E}(s,a)$, we refer to \cite{KS}.

In a recent paper \cite{HK2021}, by introducing  an operator   $T_{p}^{a}$, a $p$-adic analogue of Van Gorder's operator $T$,  we showed that the $p$-adic Hurwitz-type Euler zeta function $\zeta_{p,E}(s,a)$ satisfies
an infinite order linear differential equation
\begin{equation}\label{mainbefore}
T_{p}^{a}\left[\zeta_{p,E}(s,a)-\langle a\rangle^{1-s}\right]
=\frac{1}{s-1}\left(\langle a-1 \rangle^{1-s}-\langle a\rangle^{1-s}\right)
\end{equation}
(see  \cite[Theorem 3.5]{HK2021}).
In contrast with the complex case,  we proved that the left hand side of the above equation is convergent everywhere for $s\in\mathbb{Z}_{p}$ with $s\neq 1$ and $a\in F$ with $|a|_{p}>1,$ where $F$ is any finite extension of
 $\mathbb{Q}_{p}$ with ramification index  over $\mathbb{Q}_{p}$  less than $p-1$ (see \cite[Corollary 3.8]{HK2021}).

In this paper, we consider this problem in the positive characteristic case by following the strategy of \cite{HK2021}.
An analogue of Riemann zeta function $\zeta(s)$ in the positive characteristic  is defined by Goss in 1979 \cite{Goss}, now known as the Goss zeta function.
 By introducing  $\zeta_{\infty}(s_{0},s,a,n)$, a Hurwitz type refinement of Goss zeta function (see (\ref{newzeta})), and an infinite order linear difference operator $L$ (see (\ref{operator})),
we prove  that 
\begin{equation}\label{main3} L\left[\zeta_{\infty}\left(\frac{1}{T},s,a,0\right)\right]=\sum_{\gamma\in\mathbb{F}_{q}} \frac{1}{\langle a+\gamma\rangle^{s}}.
\end{equation}
(see Theorem \ref{Main theorem}).
It may be viewed as an analogue of (\ref{HurDE1}) and (\ref{mainbefore}) in the positive characteristic setting.

Our paper will be organized as follows. In Section 2, we first  recall the definition of the Goss zeta function and introduce a Hurwitz type refinement  of it,
then we define an infinite order linear difference operator $L$. After these,
we shall state our main result (Theorem \ref{Main theorem}). In Section 3, we will investigate
the analytic properties of the function $\zeta_{\infty}(s_{0},s,a,n)$ (see Proposition \ref{analytic1}). In Section 4, 
after proving a shifted identity for $\zeta_{\infty}(s_{0},s,a,n)$ (Lemma \ref{Lemma1}), we will get the main result.
\section{Goss zeta function and its refinement}
In this section, we first recall the definition of the Goss zeta function, then introduce a Hurwitz type refinement  of it. 

First we state some notations. Let $q=p^{k}$ be a power of a prime number $p$. Let $R=\mathbb{F}_{q}[T]$ be the polynomial ring with one variable
over the finite field $\mathbb{F}_{q}$ and $K=\mathbb{F}_{q}(T)$ be the rational function field.
For a polynomial $a\in R$, $\textrm{deg} (a)$ denotes its degree.
Let $K_{\infty}=\mathbb{F}_{q}((\frac{1}{T}))$ be the completion of $K$ at the infinite place $\infty=\left(\frac{1}{T}\right)$
and $K_{\infty}^{*}=K_{\infty}\setminus\{0\}$ be the multiplicative group of the field $K_{\infty}$. Let $\mathbb{Z}_{p}$ be the ring of $p$-adic
integers. For 
an element $a\in K_{\infty}$, it has the expansion
\begin{equation}a=a_{m}T^{m}+a_{m-1}T^{m-1}+\cdots+a_{0}+a_{-1}T^{-1}+\cdots \end{equation}
with $a_{j}\in\mathbb{F}_{q}~ (j\leq m)$ and $a_{m}\neq 0.$ Let 
\begin{equation}\textrm{sgn}_{\infty}(a)=a_{m},~~
v_{\infty}(a)=-m \end{equation}
  and the absolute value $$|a|_{\infty}=\left(\frac{1}{e}\right)^{v_{\infty}(a)}.$$
Define \begin{equation}\label{bra} \langle a \rangle=a_{m}^{-1}T^{v_{\infty}(a)} a=1+a_{m}^{-1}a_{m-1}T^{-1}+\cdots \end{equation}
and 
\begin{equation}\label{omega} \omega_{\infty}(a)=\frac{a}{\langle a \rangle}=a_{m}T^{-v_{\infty}(a)}.\end{equation}
 $\omega_{\infty}$ is an analogue of  the Teichm\"uller character in the $p$-adic analysis. 
 Furthermore let $$\mathbb{S}=K_{\infty}^{*}\times\mathbb{Z}_{p},$$
which is an analogue of the complex plane $\mathbb{C}.$ 
The set $\mathbb{S}$  has endowed with a product topology from $K_{\infty}^{*}$ and $\mathbb{Z}_{p}$. (See \cite[Definition 2.1]{Goss}).

Let $R_{1}$ be the set of monic polynomials in $R,$ 
which is an analogue of the set of positive integers. For $a\in R_{1}, w=(s_{0},s)\in \mathbb{S}=K_{\infty}^{*}\times\mathbb{Z}_{p},$
define \begin{equation} a^{w}=s_{0}^{-v_{\infty}(a)}\langle a \rangle^{s},\end{equation}
which is an analogue of $n^{s}$ in the classical case.

From this, in 1979 Goss \cite{Goss} introduced an analogue of Riemann zeta function $\zeta(s)$ in the characteristic $p$ case. In fact, he considered 
the series 
\begin{equation}\label{Goss}
\zeta_{\infty}(s_{0}, s)=\sum_{l=0}^{\infty}\sum_{a\in R_{1}\atop \textrm{deg} a=l} a^{-w}=\sum_{l=0}^{\infty}s_{0}^{-l}\sum_{a\in R_{1}\atop \textrm{deg} a=l} \frac{1}{\langle a\rangle^{s}}\end{equation}
for $w=(s_{0},s)\in \mathbb{S}=K_{\infty}^{*}\times\mathbb{Z}_{p}$ and showed that it is an entire function on $\mathbb{S}$.
We will recall the concept of entirety for a function on $\mathbb{S}$ in Section 3 (see Definition \ref{Def 3.1}).

Inspired by the definition of Hurwitz zeta functions $\zeta(s,a)$, in the following we introduce a refinement of the Goss
zeta function.
Denote by \begin{equation}\label{set} \mathbb{A}=\{a\in K_{\infty}^{*} \mid |a|_{\infty}>1\}.\end{equation}
Let $\mathbb{F}_{q}[\frac{1}{T}]$ be the polynomial ring in $\frac{1}{T}$.
If $k\in \mathbb{F}_{q}[\frac{1}{T}]$, then $\deg_{\infty}(k)$ denotes its degree as a polynomial in $\frac{1}{T}$.
Now for a fixed  $(a,n)\in\mathbb{A}\times\mathbb{Z}$ with $v_{\infty}(a)=-m$, we define the zeta function 
\begin{equation}\label{newzeta}
\zeta_{\infty}(s_{0},s,a,n)=\sum_{l=0}^{\infty} s_{0}^{-(m+l+1)n}\sum_{\substack{k\in\mathbb{F}_{q}[\frac{1}{T}]\\\textrm{deg}_{\infty}(k)\leq l}}\frac{1}{\langle k+a\rangle^{s}}
\end{equation}
on the plane $(s_{0},s)\in\mathbb{S}=K_{\infty}^{*}\times\mathbb{Z}_{p}$.

In the above construction (\ref{newzeta}), since $a\in\mathbb{A}$, we have $|a|_{\infty}>1,$ and for $k\in\mathbb{F}_{q}[\frac{1}{T}]$, we have $|k|_{\infty}\leq 1$, which correspond to
$a\in\mathbb{C}_{p}$ with $|a|_{p}>1$, and $t\in\mathbb{Z}_{p}$ thus $|t|_{p}\leq 1$ in the definition of the $p$-adic Hurwitz zeta functions
in the characteristic zero case, respectively (see (\ref{E-zeta})). These lead the applications of the binomial theorem to derive the shifted identities (\ref{main1}) and \cite[(3.1)]{HK2021}) possible. 
It may be interesting to compare the proof of Lemma \ref{Lemma1} below and the proof of Lemma 3.1 in \cite{HK2021}, especially (\ref{4.6}) and \cite[(3.4)]{HK2021}.

In the following, we shall introduce an infinite order linear difference operator $L$ from the forward difference operator $\Delta$.

Let $\{a_{n}\}_{n=1}^{\infty}$ be a sequence. The forward difference operator $\Delta$ is defined by 
$$\Delta a_{n}=a_{n+1}-a_{n}.$$
and the higher order differences may be defined recursively by
$$\Delta^{h}a_{n}=\Delta^{h-1}a_{n+1}-\Delta^{h-1}a_{n}.$$
Thus $$\Delta^{h}a_{n}=\sum_{l=0}^{h}(-1)^{l}\binom{h}{l}a_{n+h-l}$$
for $h=1,2,3,\ldots$ 
and a formula for the $(n+j)$th term is given by
\begin{equation}\label{difference}
a_{n+j}=(1+\Delta)^{j}a_{n}:=\sum_{h=0}^{j}\binom{j}{h}\Delta^{h}a_{n}\end{equation}
(see \cite[p. 10]{SP}). 

Now given a function \begin{equation}
\begin{aligned}
f: \mathbb{Z}_{p}\times \mathbb{Z} &\rightarrow K_{\infty}\\
(s,n)&\mapsto f(s,n),
\end{aligned}
\end{equation}
 define a forward difference operator $\Delta_{(s,n)}$ by
\begin{equation}\label{delta}\Delta_{(s,n)}f(s,n)=f(s+1,n+1)-f(s,n).\end{equation}
Then from (\ref{difference}) we have
\begin{equation}
f(s+j,n+j)=(1+\Delta_{(s,n)})^{j}f(s,n)
\end{equation}
for $j\in\mathbb{N}$.
With the above forward difference operator $\Delta_{(s,n)}$, we introduce an infinite order linear difference operator $L$ by
\begin{equation}\label{operator} L:= id+\sum_{\substack{j=1\\(q-1)\mid j}}^{\infty}\binom{-s}{j}(1+\Delta_{(s,n)})^{j}.\end{equation}
Now we are at the position to state our main result.
\begin{theorem}\label{Main theorem}
For $a\in\mathbb{A}$ and $s\in\mathbb{Z}_{p}$,  $\zeta_{\infty}(s_{0},s,a,n)$  satisfies the following infinite order linear difference equation 
\begin{equation}\label{main} 
L\left[\zeta_{\infty}\left(\frac{1}{T},s,a,0\right)\right]=\sum_{\gamma\in\mathbb{F}_{q}} \frac{1}{\langle a+\gamma\rangle^{s}}.
\end{equation}
\end{theorem}

\section{The entirety of  $\zeta_{\infty}(s_{0},s,a,n)$}
In this section, to investigate the analytic properties of  $\zeta_{\infty}(s_{0},s,a,n)$,
we first recall the concept of entirety for a function on $\mathbb{S}=K_{\infty}^{*}\times\mathbb{Z}_{p}$ according to Goss's book \cite[p. 248--249]{Gossbook}.
\begin{definition}[{see \cite[p. 248, Definition 8.5.1]{Gossbook}}] \label{Def 3.1}
We define an entire function $f(w)=f(s_{0},s)$ on $\mathbb{S}$ to be a continuous family of $K_{\infty}$-valued entire power series in $s_{0}^{-1}$ 
which is parametrized by $\mathbb{Z}_{p}$. Moreover, this family is required to be uniformly convergent on bounded subsets of $K_{\infty}$.
\end{definition}
The above definition means that a function $f(w)=f(s_{0},s): \mathbb{S}\rightarrow K_{\infty}$ is said to be entire if and only if  
we have the following power series expansion
$$f(w)=f(s_{0},s)=\sum_{l=0}^{\infty}f_{l}(s)s_{0}^{-l},$$
where $f_{l}(s): \mathbb{Z}_{p}\rightarrow K_{\infty}$ is continuous, and the above series is uniformly convergent for $s_{0}$ in any bounded subset of $K_{\infty}$.

In Proposition \ref{analytic1} below, we will consider the entirety of  $\zeta_{\infty}(s_{0},s,a,n)$. First, we need the following lemma.
\begin{lemma}\label{analytic}
For a fixed $\alpha\in K_{\infty}^{*}$, $f(s)=\langle \alpha \rangle^{s}$ is a continuous function from $\mathbb{Z}_{p}$ to $K_{\infty}$.
\end{lemma}
\begin{proof} For $\alpha\in K_{\infty}^{*}$ we have $$\langle \alpha \rangle=1+\lambda_{\alpha}$$
with $\left|\lambda_{\alpha}\right|_{\infty}< 1.$
Then by \cite[p. 237, Definition 8.1.2]{Gossbook},  for $s\in\mathbb{Z}_{p}$
 \begin{equation}\label{expa}
\langle \alpha \rangle^{s}=\sum_{j=0}^{\infty}\binom{s}{j}\lambda_{\alpha}^{j}.\end{equation}
Here
\begin{equation*}
\binom{s}{j}=\frac{s(s-1)\cdots (s-j+1)}{j!}
\end{equation*}
gives a continuous function from $\mathbb{Z}_{p}$ to $\mathbb{Z}_{p}$.
By reducing it modulo $(p)$, we can consider it as a continuous function with values in $\mathbb{F}_{p}\subset\mathbb{F}_{q}\subset K_{\infty}$
(see \cite[p. 245, Definition 8.4.2]{Gossbook}).
So 
\begin{equation}\label{bino}
\left|\binom{s}{j}\right|_{\infty}\leq 1
\end{equation} and 
 \begin{equation}
 \left|\langle \alpha \rangle^{s}\right|_{\infty}\leq \sum_{j=0}^{\infty}\left|\lambda_{\alpha}\right|_{\infty}^{j}.
   \end{equation}
Since for a fixed  $\langle \alpha \rangle=1+\lambda_{\alpha}$ with $\left|\lambda_{\alpha}\right|_{\infty}<1$,
from the Weierstrass test, the power series (\ref{expa}) is uniformly convergent  for $s\in\mathbb{Z}_{p}$.
By \cite[p. 182, Theorem 3.2]{Lang2},
$f(s)=\langle \alpha \rangle^{s}$  is a continuous function from $\mathbb{Z}_{p}$ to $K_{\infty}$.
   \end{proof}
\begin{proposition}[{The entirety of  $\zeta_{\infty}(s_{0},s,a,n)$}]\label{analytic1}
Fixed $(a,n)\in\mathbb{A}\times\mathbb{Z}$, the function $\zeta_{\infty}(s_{0},s,a,n)$ is  an entire function for $(s_{0},s)\in\mathbb{S}=K_{\infty}^{*}\times\mathbb{Z}_{p}$.
\end{proposition}
\begin{proof}
For $k\in\mathbb{F}_{q}[\frac{1}{T}]$ with $ \textrm{deg}_{\infty}(k)\leq l$,
we have the expansion $$k=b_{-l}T^{-l}+b_{-l+1}T^{-l+1}+\cdots+b_{-1}T^{-1}+b_{0}$$
with $b_{j}\in\mathbb{F}_{q}~ (-l\leq j\leq 0).$
And for $a\in\mathbb{A}$ so $v_{\infty}(a)=-m< 0$ (see (\ref{set})), we have the expansion 
$$a=a_{m}T^{m}+a_{m-1}T^{m-1}+\cdots+a_{0}+a_{-1}T^{-1}+\cdots $$
with $a_{j}\in\mathbb{F}_{q}~ (j\leq m)$ and $a_{m}\neq 0.$
Thus by (\ref{bra}) we have
\begin{equation}\label{quan1}
\begin{aligned}
\langle k+a \rangle&=a_{m}^{-1}b_{-l}T^{-m-l}+a_{m}^{-1}b_{-l+1}T^{-m-l+1}+\cdots+a_{m}^{-1}b_{0}T^{-m}\\
&\quad+1+a_{m}^{-1}a_{m-1}T^{-1}+\cdots+a_{m}^{-1}a_{0}T^{-m}+a_{m}^{-1}a_{-1}T^{-m-1}+\cdots\\
&=(1+a_{m}^{-1}a_{m-1}T^{-1}+\cdots+a_{m}^{-1}a_{0}T^{-m}+a_{m}^{-1}a_{-1}T^{-m-1}+\cdots)\\
&\quad+(a_{m}^{-1}b_{-l}T^{-m-l}+a_{m}^{-1}b_{-l+1}T^{-m-l+1}+\cdots+a_{m}^{-1}b_{0}T^{-m}).
\end{aligned}
\end{equation}
Let $$x=1+a_{m}^{-1}a_{m-1}T^{-1}+\cdots+a_{m}^{-1}a_{0}T^{-m}+a_{m}^{-1}a_{-1}T^{-m-1}+\cdots$$
and $$w=a_{m}^{-1}b_{-l}T^{-m-l}+a_{m}^{-1}b_{-l+1}T^{-m-l+1}+\cdots+a_{m}^{-1}b_{0}T^{-m}.$$
So by (\ref{quan1}) we have
\begin{equation}\label{quan2}
\langle k+a \rangle=x+w.
\end{equation}
Note that if $k$ varies over the set
\begin{equation}\label{dis}
C=\left\{k\in\mathbb{F}_{q}[\frac{1}{T}]\mid \textrm{deg}_{\infty}(k)\leq l\right\},
\end{equation}
 then $w$ will run through the space
\begin{equation}\label{dis1}
W=\left\{\beta_{l}T^{-m-l}+\beta_{l-1}T^{-m-l+1}+\cdots+\beta_{0}T^{-m}\mid \beta_{0}, \beta_{1}, \ldots, \beta_{l}\in\mathbb{F}_{q} \right\}.
\end{equation}
 Since  the function 
\begin{equation} 
\begin{aligned} 
\mathbb{Z}_{p}&\to K_{\infty}\\
                                                     j&\mapsto \langle k+a\rangle^{j}
\end{aligned}
\end{equation}
is continuous and the natural number system $\mathbb{N}$ is dense in $\mathbb{Z}_{p}$,
we only need to consider the case for $j\in\mathbb{N}$.
By (\ref{quan2}), for $j\in\mathbb{N}$ we have 
\begin{equation}\label{quan3}
\begin{aligned}
\sum_{\substack{k\in\mathbb{F}_{q}[\frac{1}{T}]\\ \textrm{deg}_{\infty}(k)\leq l}}\langle k+a\rangle^{j}
&=\sum_{w\in W}(x+w)^{j}\\
&=\sum_{\beta_{0}, \beta_{1}, \cdots, \beta_{l}\in\mathbb{F}_{q}}(x+\beta_{l}T^{-m-l}+\beta_{l-1}T^{-m-l+1}+\cdots+\beta_{0}T^{-m})^{j}\\
&=\sum_{\substack{0\leq j_{-1}, j_{0},j_{1},\ldots,j_{l}\leq j\\j_{-1}+ j_{0}+j_{1}+\cdots+j_{l}=j}}\binom{j}{j_{-1}, j_{0},j_{1},\ldots,j_{l}}x^{j_{-1}}\\
&\quad\times\sum_{\beta_{l}\in\mathbb{F}_{q}}(\beta_{l}T^{-m-l})^{j_l}\sum_{\beta_{l-1}\in\mathbb{F}_{q}}(\beta_{l-1}T^{-m-l+1})^{j_{l-1}}\cdots\sum_{\beta_{0}\in\mathbb{F}_{q}}(\beta_{0}T^{-m})^{j_0}\\
&=\sum_{\substack{0\leq j_{-1}, j_{0},j_{1},\ldots,j_{l}\leq j\\j_{-1}+ j_{0}+j_{1}+\cdots+j_{l}=j}}\binom{j}{j_{-1}, j_{0},j_{1},\ldots,j_{l}}x^{j_{-1}}\\
&\quad\times\left(\sum_{\beta_{l}\in\mathbb{F}_{q}}\beta_{l}^{j_{l}}\right)T^{(-m-l)j_{l}}\left(\sum_{\beta_{l-1}\in\mathbb{F}_{q}}\beta_{l-1}^{j_{l-1}}\right)T^{(-m-l+1)j_{l-1}}\cdots\left(\sum_{\beta_{0}\in\mathbb{F}_{q}}\beta_{0}^{j_{0}}\right)T^{(-m)j_{0}},
\end{aligned}
\end{equation}
where $$\binom{j}{j_{-1}, j_{0},j_{1},\ldots,j_{l}}=\frac{j!}{j_{-1}!j_{0}!\cdots j_{l}!}$$
and $0^{0}=1$ by convention.

For $j_{h}<(q-1)$, we have $$\left(\sum_{\beta_{j}\in\mathbb{F}_{q}}\beta_{j}^{j_{h}}\right)T^{(-m-j)j_{h}}=0.$$
For $j_{h}\geq q-1$, we have
\begin{equation}\label{quan4}
\begin{aligned}
v_{\infty}\left(\left(\sum_{\beta_{j}\in\mathbb{F}_{q}}\beta_{j}^{j_{h}}\right)T^{(-m-j)j_{h}}\right)&\geq v_{\infty}\left(T^{(-m-j)j_{h}}\right)\\
&=(m+j)j_{h}\\&\geq (q-1)(m+j)\end{aligned}
\end{equation}
and \begin{equation}\label{quan4add}
\left|\left(\sum_{\beta_{j}\in\mathbb{F}_{q}}\beta_{j}^{j_{h}}\right)T^{(-m-j)j_{h}}\right|_{\infty}
\leq \left(\frac{1}{e}\right)^{(q-1)(m+j)}.
\end{equation}
Since the valuations
$$\left|\binom{j}{j_{-1}, j_{0},j_{1},\ldots,j_{l}}\right|_{\infty}\leq 1 ~~\textrm{and}~~\left|x^{j_{-1}}\right|_{\infty}=1,$$
by  (\ref{quan3}) and (\ref{quan4add}) we have
\begin{equation}\label{quan5}
\begin{aligned}
\left|\sum_{\substack{k\in\mathbb{F}_{q}[\frac{1}{T}]\\ \textrm{deg}_{\infty}(k)\leq l}}\langle k+a\rangle^{j}\right|_{\infty}&\leq \left(\frac{1}{e}\right)^{(q-1)\sum_{j=0}^{l}(m+j)}\\
&=\left(\frac{1}{e}\right)^{(q-1)\frac{(2m+l)(l+1)}{2}}.
\end{aligned}
\end{equation}
Thus for any $s\in\mathbb{Z}_{p}$ 
\begin{equation}\label{quan6}
 \left|\sum_{\substack{k\in\mathbb{F}_{q}[\frac{1}{T}]\\\textrm{deg}_{\infty}(k)\leq l}}\frac{1}{\langle k+a\rangle^{s}}\right|_{\infty}
\leq\left(\frac{1}{e}\right)^{(q-1)\frac{(2m+l)(l+1)}{2}}.
\end{equation}
Notice that \begin{equation}\label{quan7}
 v_{\infty}(s_{0}^{-(m+l+1)n})=-(m+l+1)nv_{\infty}(s_{0})\end{equation}
and \begin{equation}\label{quan8}
\left|s_{0}^{-(m+l+1)n}\right|_{\infty}=\left(\frac{1}{e}\right)^{-(m+l+1)n v_{\infty}(s_{0})}.\end{equation}
Combing (\ref{quan6}) and (\ref{quan8}) we have
\begin{equation}
\left|s_{0}^{-(m+l+1)n}\sum_{\substack{k\in\mathbb{F}_{q}[\frac{1}{T}]\\ \textrm{deg}_{\infty}(k)\leq l}}\frac{1}{\langle k+a\rangle^{s}}\right|_{\infty}
\leq \left(\frac{1}{e}\right)^{(q-1)\frac{(2m+l)(l+1)}{2}-(m+l+1)nv_{\infty}(s_{0})}.
\end{equation}
So 
\begin{equation}
\lim_{l\to\infty}\left|s_{0}^{-(m+l+1)n}\sum_{\substack{k\in\mathbb{F}_{q}[\frac{1}{T}]\\ \textrm{deg}_{\infty}(k)\leq l}}\frac{1}{\langle k+a\rangle^{s}}\right|_{\infty}=0
\end{equation}
uniformly  for $s_{0}$ in any bounded subset of $K_{\infty}$.
Thus for any fixed $(a,n)\in\mathbb{A}\times\mathbb{Z}$ with $v_{\infty}(a)=-m$, the series 
\begin{equation} 
\zeta_{\infty}(s_{0},s,a,n)=\sum_{l=0}^{\infty} s_{0}^{-(m+l+1)n}\sum_{\substack{k\in\mathbb{F}_{q}[\frac{1}{T}]\\\textrm{deg}_{\infty}(k)\leq l}}\frac{1}{\langle k+a\rangle^{s}}
\end{equation} is uniformly convergent for $s_{0}$ in any bounded subset of $K_{\infty}$.
So by Definition \ref{Def 3.1} and  Lemma \ref{analytic}, we see that $\zeta_{\infty}(s_{0},s,a,n)$ is  an entire function for $(s_{0},s)\in\mathbb{S}=K_{\infty}^{*}\times\mathbb{Z}_{p}$.
 \end{proof}

\section{Proof of the main result}
In this section, we shall prove Theorem \ref{Main theorem}, which is implied by the following shifted identity for $\zeta_{\infty}(s_{0},s,a,n)$.

\begin{lemma}\label{Lemma1}
For $a\in\mathbb{A}$ and $s\in\mathbb{Z}_{p}$,  $\zeta_{\infty}(s_{0},s,a,n)$  satisfies the following identity
\begin{equation}\label{main1}\sum_{\gamma\in\mathbb{F}_{q}} \frac{1}{\langle a+\gamma\rangle^{s}}=\zeta_{\infty}\left(\frac{1}{T},s,a,0\right)+\sum_{\substack{j=1\\(q-1)\mid j}}^{\infty}\binom{-s}{j}\zeta_{\infty}\left(\frac{1}{T},s+j,a,j\right).
\end{equation}
\end{lemma}

\begin{remark}
This is a positive characteristic analogue of two known identities in the characteristic zero case. One is  the complex identity for Hurwitz zeta function $\zeta(s,a)$ (see \cite[Lemma 1]{PK}), the other is the $p$-adic identity for  $p$-adic Hurwitz-type Euler zeta function $\zeta_{p,E}(s,a)$ (see \cite[Lemma 3.1]{HK2021}).
\end{remark}
\begin{remark} It may be interesting to compare  (\ref{main1}) with the following recurrence relation for the special values of the Goss zeta function
$$\zeta_{\infty}(-N)=1-\sum_{\substack{j=0 \\ (q-1)\mid (N-j)}}^{N-1}\binom{N}{j}T^{j}\zeta_{\infty}(-j)$$
for $N\in\mathbb{N}$ (see \cite[Theorem 5.6]{Goss}).
\end{remark}

\begin{proof}[Proof of Lemma \ref{Lemma1}.]
For  $s\in\mathbb{Z}_{p}$ and $N\in\mathbb{N}$ we have
 \begin{equation}\label{star}
\sum_{\gamma\in\mathbb{F}_{q}} \frac{1}{\langle a+\gamma\rangle^{s}}=\sum_{l=0}^{N}\sum_{\substack{k\in\mathbb{F}_{q}[\frac{1}{T}]\\ \textrm{deg}_{\infty}(k)\leq l}} \frac{1}{\langle k+a \rangle^{s}}-\sum_{l=1}^{N}\sum_{\substack{k\in\mathbb{F}_{q}[\frac{1}{T}]\\ \textrm{deg}_{\infty}(k)\leq l}} \frac{1}{\langle k+a \rangle^{s}}.  \end{equation}
Denote by \begin{equation}
\widetilde{W}=\left\{\beta_{l-1}T^{-m-l+1}+\cdots+\beta_{0}T^{-m}\mid \beta_{0}, \beta_{1}, \ldots, \beta_{l-1}\in\mathbb{F}_{q} \right\}.
\end{equation} 
Obviously, $$\dim_{\mathbb{F}_{q}}\widetilde{W}=\dim_{\mathbb{F}_{q}}W-1=(l+1)-1=l.$$
By Eqs. (\ref{quan2}), (\ref{dis}) and (\ref{dis1}) we have
\begin{equation}\label{lquan1}
 \begin{aligned}
 \sum_{\substack{k\in\mathbb{F}_{q}[\frac{1}{T}]\\ \textrm{deg}_{\infty}(k)\leq l}} \frac{1}{\langle k+a \rangle^{s}}=\sum_{\substack{w\in W\\\dim_{\mathbb{F}_{q}}W=l+1}}\frac{1}{(x+w)^{s}}
 =\sum_{\substack{\tilde{w}\in \widetilde{W}\\\dim_{\mathbb{F}_{q}}\widetilde{W}=l}}\sum_{\beta_{l}\in\mathbb{F}_{q}}\frac{1}{(x+\tilde{w}+\beta_{l}T^{-m-l})^{s}}.  \end{aligned}
 \end{equation}
So (\ref{star}) implies that
 \begin{equation}\label{lquan2}
 \begin{aligned}
\sum_{\gamma\in\mathbb{F}_{q}} \frac{1}{\langle a+\gamma\rangle^{s}}=&\sum_{l=0}^{N}\sum_{\substack{w\in W\\\dim_{\mathbb{F}_{q}}W=l+1}}\frac{1}{(x+w)^{s}}-\sum_{l=1}^{N}\sum_{\substack{w\in W\\\dim_{\mathbb{F}_{q}}W=l+1}}\frac{1}{(x+w)^{s}}\\
&=\sum_{l=0}^{N}\sum_{\substack{w\in W\\\dim_{\mathbb{F}_{q}}W=l+1}}\frac{1}{(x+w)^{s}}-\sum_{l=1}^{N}\sum_{\substack{\tilde{w}\in \widetilde{W}\\\dim_{\mathbb{F}_{q}}\widetilde{W}=l}}\sum_{\beta_{l}\in\mathbb{F}_{q}}\frac{1}{(x+\tilde{w}+\beta_{l}T^{-m-l})^{s}}  \\
&=\sum_{l=0}^{N}\sum_{\substack{w\in W\\\dim_{\mathbb{F}_{q}}W=l+1}}\frac{1}{(x+w)^{s}}-\sum_{l=0}^{N-1}\sum_{\substack{w\in W\\\dim_{\mathbb{F}_{q}}W=l+1}}\sum_{\beta_{l+1}\in\mathbb{F}_{q}}\frac{1}{(x+w+\beta_{l+1}T^{-m-l-1})^{s}}  \\
&=\sum_{l=0}^{N-1}\sum_{\substack{w\in W\\\dim_{\mathbb{F}_{q}}W=l+1}}\left(\frac{1}{(x+w)^{s}}-\sum_{\beta_{l+1}\in\mathbb{F}_{q}}\frac{1}{(x+w+\beta_{l+1}T^{-m-l-1})^{s}}\right)\\
&\quad+\sum_{\substack{w\in W\\\dim_{\mathbb{F}_{q}}W=N+1}}\frac{1}{(x+w)^{s}}\\
&=\sum_{l=0}^{N-1}\sum_{\substack{w\in W\\\dim_{\mathbb{F}_{q}}W=l+1}}\frac{1}{(x+w)^{s}}\left(1-\sum_{\beta_{l+1}\in\mathbb{F}_{q}}\left(\frac{x+w}{x+w+\beta_{l+1}T^{-m-l-1}}\right)^{s}\right)\\
&\quad+\sum_{\substack{k\in\mathbb{F}_{q}[\frac{1}{T}]\\ \textrm{deg}_{\infty}(k)\leq N}} \frac{1}{\langle k+a \rangle^{s}}.
 \end{aligned}
 \end{equation}
Since $x+w$ is a $1$-unit, and $|a|_{\infty}>1$ so $v_{\infty}(a)=-m<0$, we have
$$\left|\frac{\beta_{l+1}T^{-m-l-1}}{x+w}\right|_{\infty}<1$$ 
for $l\in\mathbb{N}\cup\{0\}$.
So \begin{equation}\label{4.6}
\begin{aligned}
\left(\frac{x+w}{x+w+\beta_{l+1}T^{-m-l-1}}\right)^{s}&=\left(1+\frac{\beta_{l+1}T^{-m-l-1}}{x+w}\right)^{-s}\\
&=\sum_{j=0}^{\infty}\binom{-s}{j}\beta_{l+1}^{j}\left(\frac{1}{T}\right)^{(m+l+1)j}\frac{1}{(x+w)^{j}}\end{aligned}
\end{equation}
and 
 \begin{equation}\label{add}
\begin{aligned}
\sum_{\beta_{l+1}\in\mathbb{F}_{q}} \left(\frac{x+w}{x+w+\beta_{l+1}T^{-m-l-1}}\right)^{s}&=\sum_{\beta_{l+1}\in\mathbb{F}_{q}}\left(1+\frac{\beta_{l+1}T^{-m-l-1}}{x+w}\right)^{-s}\\
&=\sum_{j=0}^{\infty}\binom{-s}{j}\left(\sum_{\beta_{l+1}\in\mathbb{F}_{q}} \beta_{l+1}^{j}\right)\left(\frac{1}{T}\right)^{(m+l+1)j}\frac{1}{(x+w)^{j}}.
\end{aligned}
\end{equation} 
 Note that $$\sum_{\beta_{l+1}\in\mathbb{F}_{q}} \beta_{l+1}^{j}=0~\textrm{if}~~(q-1)\nmid j~ \textrm{or}~j=0$$
 and  $$\sum_{\beta_{l+1}\in\mathbb{F}_{q}} \beta_{l+1}^{j}=q-1=-1~\textrm{if}~~(q-1)\mid j~ \textrm{and}~j\geq 1, $$
(\ref{add}) implies 
  \begin{equation}
\begin{aligned}
\sum_{\beta_{l+1}\in\mathbb{F}_{q}} \left(\frac{x+w}{x+w+\beta_{l+1}T^{-m-l-1}}\right)^{s}
=-\sum_{\substack{j=1\\(q-1)\mid j}}^{\infty}\binom{-s}{j}\left(\frac{1}{T}\right)^{(m+l+1)j}\frac{1}{(x+w)^{j}}.
\end{aligned}
\end{equation}  
 Substituting the above expansion into (\ref{lquan2}), we have
 \begin{equation}\label{lquan4}
 \begin{aligned}
\sum_{\gamma\in\mathbb{F}_{q}} \frac{1}{\langle a+\gamma\rangle^{s}}
 &=\sum_{l=0}^{N-1}\sum_{\substack{w\in W\\\dim_{\mathbb{F}_{q}}W=l+1}}\frac{1}{(x+w)^{s}}\left(1+\sum_{\substack{j=1\\(q-1)\mid j}}^{\infty}\binom{-s}{j}\left(\frac{1}{T}\right)^{(m+l+1)j}\frac{1}{(x+w)^{j}}\right)\\
&\quad+\sum_{\substack{k\in\mathbb{F}_{q}[\frac{1}{T}]\\ \textrm{deg}_{\infty}(k)\leq N}} \frac{1}{\langle k+a \rangle^{s}}\\
&=\sum_{l=0}^{N-1}\sum_{\substack{w\in W\\\dim_{\mathbb{F}_{q}}W=l+1}}\frac{1}{(x+w)^{s}}+\sum_{\substack{j=1\\(q-1)\mid j}}^{\infty}\binom{-s}{j}\sum_{l=0}^{N-1}\left(\frac{1}{T}\right)^{(m+l+1)j}\sum_{\substack{w\in W\\\dim_{\mathbb{F}_{q}}W=l+1}}\frac{1}{(x+w)^{s+j}}\\
&\quad+\sum_{\substack{k\in\mathbb{F}_{q}[\frac{1}{T}]\\ \textrm{deg}_{\infty}(k)\leq N}} \frac{1}{\langle k+a \rangle^{s}}\\
&=\sum_{l=0}^{N-1} \sum_{\substack{k\in\mathbb{F}_{q}[\frac{1}{T}]\\ \textrm{deg}_{\infty}(k)\leq l}} \frac{1}{\langle k+a \rangle^{s}}+\sum_{\substack{j=1\\(q-1)\mid j}}^{\infty}\binom{-s}{j}\sum_{l=0}^{N-1}\left(\frac{1}{T}\right)^{(m+l+1)j}\sum_{\substack{k\in\mathbb{F}_{q}[\frac{1}{T}]\\ \textrm{deg}_{\infty}(k)\leq l}} \frac{1}{\langle k+a \rangle^{s+j}}\\
&\quad+\sum_{\substack{k\in\mathbb{F}_{q}[\frac{1}{T}]\\ \textrm{deg}_{\infty}(k)\leq N}} \frac{1}{\langle k+a \rangle^{s}}. \end{aligned}
 \end{equation}
 Taking $N\to\infty$ in the above equality, by (\ref{quan6}), we have $$\lim_{N\to\infty}\left|\sum_{\substack{k\in\mathbb{F}_{q}[\frac{1}{T}]\\ \textrm{deg}_{\infty}(k)\leq N}}\frac{1}{\langle k+a \rangle^{s}}\right|_{\infty}=0$$ and
 \begin{equation}\label{lquan6}
 \begin{aligned}
\sum_{\gamma\in\mathbb{F}_{q}} \frac{1}{\langle a+\gamma\rangle^{s}}
 &=\lim_{N\to\infty}\sum_{l=0}^{N-1}\sum_{\substack{k\in\mathbb{F}_{q}[\frac{1}{T}]\\ \textrm{deg}_{\infty}(k)\leq l}}\frac{1}{\langle k+a \rangle^{s}}\\
 &\quad+\sum_{\substack{j=1\\(q-1)\mid j}}^{\infty}\binom{-s}{j}\lim_{N\to\infty}\sum_{l=0}^{N-1}\left(\frac{1}{T}\right)^{(m+l+1)j}\sum_{\substack{k\in\mathbb{F}_{q}[\frac{1}{T}]\\ \textrm{deg}_{\infty}(k)\leq l}} \frac{1}{\langle k+a \rangle^{s+j}}\\
 &\quad (\textrm{see Proposition \ref{proposition-add}})\\
 &=\sum_{l=0}^{\infty}\sum_{\substack{k\in\mathbb{F}_{q}[\frac{1}{T}]\\ \textrm{deg}_{\infty}(k)\leq l}}\frac{1}{\langle k+a \rangle^{s}}\\
 &\quad+\sum_{\substack{j=1\\(q-1)\mid j}}^{\infty}\binom{-s}{j}\sum_{l=0}^{\infty}\left(\frac{1}{T}\right)^{(m+l+1)j}\sum_{\substack{k\in\mathbb{F}_{q}[\frac{1}{T}]\\ \textrm{deg}_{\infty}(k)\leq l}} \frac{1}{\langle k+a \rangle^{s+j}}.
 \end{aligned}
 \end{equation} 
Then by the definitions of the zeta function $\zeta_{\infty}(s_{0},s,a,n)$ (\ref{newzeta}), we get
\begin{equation}\label{lquan7}
\sum_{\gamma\in\mathbb{F}_{q}} \frac{1}{\langle a+\gamma\rangle^{s}}=\zeta_{\infty}\left(\frac{1}{T},s,a,0\right)+\sum_{\substack{j=1\\(q-1)\mid j}}^{\infty}\binom{-s}{j}\zeta_{\infty}\left(\frac{1}{T},s+j,a,j\right).\end{equation} 
This completes our proof.
\end{proof}

In order to move the limit to the inside of the
summation $\displaystyle\sum_{\substack{j=1\\(q-1)\mid j}}^{\infty}$ in (\ref{lquan6}) of the above lemma, we need to show that the convergence of the inner
limit is uniform  for $N\in\mathbb{N}$. To this end, we add the following proposition.

\begin{proposition}\label{proposition-add}
For $a\in\mathbb{A}$ and $s\in\mathbb{Z}_{p}$, the series $$\sum_{\substack{j=1\\(q-1)\mid j}}^{\infty}\binom{-s}{j}\sum_{l=0}^{N-1} T^{-(m+l+1)j}\sum_{\substack{k\in\mathbb{F}_{q}[\frac{1}{T}]\\ {\rm deg}_{\infty}(k)\leq l}}\frac{1}{\langle k+a \rangle^{s+j}}$$ 
is uniformly convergent for $N\in\mathbb{N}$ and 
\begin{equation}\begin{aligned}&\quad\lim_{N\to\infty}\sum_{\substack{j=1\\(q-1)\mid j}}^{\infty}\binom{-s}{j}\sum_{l=0}^{N-1} T^{-(m+l+1)j}\sum_{\substack{k\in\mathbb{F}_{q}[\frac{1}{T}]\\ {\rm deg}_{\infty}(k)\leq l}}\frac{1}{\langle k+a \rangle^{s+j}}\\
&=\sum_{\substack{j=1\\(q-1)\mid j}}^{\infty}\binom{-s}{j}\lim_{N\to\infty}\sum_{l=0}^{N-1} T^{-(m+l+1)j}\sum_{\substack{k\in\mathbb{F}_{q}[\frac{1}{T}]\\ {\rm deg}_{\infty}(k)\leq l}}\frac{1}{\langle k+a \rangle^{s+j}}.\end{aligned} \end{equation}\end{proposition}

\begin{proof}
Let $$U=1+\frac{1}{T}\mathbb{F}_{q}[[\frac{1}{T}]]$$ 
be the group of $1$-units in $K_{\infty}$.
By \cite[p. 98, the second paragraph]{Goss2}, for any $u=1+\omega\in U$ with $|\omega|_{\infty}<1$ and $s=\sum_{j=j_{0}}^{\infty} c_{j}p^{j} \in \mathbb{Z}_{p}$, we have
\begin{equation}
u^{s}:=\prod_{j}(1+\omega^{p^{j}})^{c_{j}}.
\end{equation}
Thus the map
\begin{equation}
\begin{aligned}
f: U \times \mathbb{Z}_{p}&\rightarrow K_{\infty}^{*}\\
(u,s)&\mapsto u^{s}
\end{aligned}
\end{equation}
is continuous.  Since the sets $U$ and $\mathbb{Z}_{p}$ are compact, the function $f(u,s)=u^{s}$ is uniformly bounded for 
$(u,s)\in U \times \mathbb{Z}_{p},$ that is, there exists a constant $M>0$ such that
\begin{equation}\label{bound}
|u^{s}|_{\infty}\leq M
\end{equation}
for $u\in U$ and $s\in\mathbb{Z}_{p}$.

So for any $k\in\mathbb{F}_{q}[\frac{1}{T}]$, $s\in\mathbb{Z}_{p}$ and $j\in\mathbb{N}$
\begin{equation}
\left|\frac{1}{\langle k+a \rangle^{s+j}}\right|_{\infty}\leq M
\end{equation}
and for  any $l\in\mathbb{N}\cup\{0\}$, $j\in\mathbb{N}$ and $s\in\mathbb{Z}_{p}$ we have
\begin{equation}\label{pquan4}
\begin{aligned}
\left|T^{-(m+l+1)j}\sum_{\substack{k\in\mathbb{F}_{q}[\frac{1}{T}]\\ \textrm{deg}_{\infty}(k)\leq l}}\frac{1}{\langle k+a \rangle^{s+j}}\right|_{\infty}&\leq \left(\frac{1}{e}\right)^{(m+l+1)j}M.
\end{aligned}
\end{equation}
Then combining the estimations (\ref{bino}) and (\ref{pquan4}) we see that
\begin{equation}\label{pquan5}
\begin{aligned}
&\quad\left|\binom{-s}{j} \sum_{l=0}^{N-1} T^{-(m+l+1)j}\sum_{\substack{k\in\mathbb{F}_{q}[\frac{1}{T}]\\ \textrm{deg}_{\infty}(k)\leq l}}\frac{1}{\langle k+a \rangle^{s+j}}\right|_{\infty}\\
&=\left|\binom{-s}{j}\right|_{\infty}\left|\sum_{l=0}^{N-1} T^{-(m+l+1)j}\sum_{\substack{k\in\mathbb{F}_{q}[\frac{1}{T}]\\ \textrm{deg}_{\infty}(k)\leq l}}\frac{1}{\langle k+a \rangle^{s+j}}\right|_{\infty}\\
&\leq M\left(\frac{1}{e}\right)^{(m+l+1)j}.
\end{aligned}
\end{equation}
The limit 
$$\lim_{j\to\infty}M\left(\frac{1}{e}\right)^{(m+l+1)j}=0$$
implies that the series 
\begin{equation}\label{pquan6} M\sum_{j=1}^{\infty}\left(\frac{1}{e}\right)^{(m+l+1)j}\end{equation}
is covergent.
Finally by (\ref{pquan5}), (\ref{pquan6}) and the Weierstrass test (see \cite[p. 230, Theorem 5.1]{Lang2}), we see that the series 
$$\sum_{\substack{j=1\\(q-1)\mid j}}^{\infty}\binom{-s}{j}\sum_{l=0}^{N-1} T^{-(m+l+1)j}\sum_{\substack{k\in\mathbb{F}_{q}[\frac{1}{T}]\\ \textrm{deg}_{\infty}(k)\leq l}}\frac{1}{\langle k+a \rangle^{s+j}}$$ 
is uniformly convergent for  $N\in\mathbb{N}$. Then applying \cite[p. 185, Theorem 3.5]{Lang2} we conclude that  the limit $N\to\infty$  can be moved to the inside of the
above series, which is the desired result. 
\end{proof}
\begin{proof}[Proof of Theorem \ref{Main theorem}.]
With the definition of the forward difference operator $\Delta_{(s,n)}$ (see (\ref{delta})), the shifted identity (\ref{main1}) can be rewritten as
\begin{equation}\label{main12}
\sum_{\gamma\in\mathbb{F}_{q}} \frac{1}{\langle a+\gamma\rangle^{s}}=\zeta_{\infty}\left(\frac{1}{T},s,a,0\right)+\sum_{\substack{j=1\\(q-1)\mid j}}^{\infty}\binom{-s}{j}(1+\Delta_{(s,n)})^{j}\zeta_{\infty}\left(\frac{1}{T},s,a,0\right),
\end{equation}
Then from the definition of  the infinite order linear difference operator $L$ (see (\ref{operator})), it is equivalent to
\begin{equation}\label{main*} 
L\left[\zeta_{\infty}\left(\frac{1}{T},s,a,0\right)\right]=\sum_{\gamma\in\mathbb{F}_{q}} \frac{1}{\langle a+\gamma\rangle^{s}},
\end{equation}
which is the desired result.
\end{proof}

\section*{Acknowledgements} 
The authors are enormously grateful to the anonymous referee for his/her very careful
reading of this paper, and for his/her many valuable and detailed suggestions. 

Su Hu is supported by is supported by the Natural Science Foundation of Guangdong Province, China (No. 2024A1515012337).  Min-Soo Kim is supported by the National Research Foundation of Korea(NRF) grant funded by the Korea government(MSIT) (No. NRF-2022R1F1A1065551).

\bibliography{central}

\end{document}